\documentclass{article}

\usepackage{amssymb}
\usepackage{amsmath}
\usepackage{amsthm}
\usepackage{enumerate}

\usepackage{graphicx}
\usepackage [applemac] {inputenc}
\usepackage{geometry}      
\usepackage[francais]{babel}  
\usepackage{color,graphics}
\usepackage[T1]{fontenc}
\usepackage{cite}    
\usepackage[all]{xy}
\usepackage{graphicx}
\usepackage{bibentry}
\usepackage{dsfont}

\usepackage{tikz}
\usepackage{float}

\title{Fonction complexité associée à une application ergodique du tore}

\author{{\Large Jean-François Bertazzon }\\
 {\it \small Laboratoire d’Analyse, Topologie et Probabilités,} \\ 
  {\it \small Aix-Marseille Université,}
   \\ {\it \small Avenue de l'escadrille Normandie-Niémen. 13397  Marseille, France} }

\begin{document}

\newcommand{\floor}[1]{{\left\lfloor #1 \right\rfloor}}

\newtheorem{theoreme}{Théorème}[section] 
\newtheorem{lemme}[theoreme]{Lemme}   
\newtheorem{proposition}[theoreme]{Proposition}

\newcounter{rmq} \setcounter{rmq}{0}
\newenvironment{remarque}{\refstepcounter{rmq}\noindent{\textsc{\small Remarque \thermq.}}}{} 

\renewcommand{\proofname}{\textup{Preuve}}

\maketitle

\renewcommand{\abstractname}{\textup{\textsc{Abstract}}}
\begin{abstract}
In this article we give the optimal lower bound for the complexity function of a planar translation which induces an ergodic rotation of the torus $\mathbb{R}^2 / \mathbb{Z}^2$.  In addition, we give an explicit calculation of this complexity for the application $(x,y) \mapsto (x+1/\phi^2,y+x-1/ (2\phi^3))$, where $\phi$ is the golden mean.
\end{abstract}

\renewcommand{\abstractname}{\textup{\textsc{Résumé}}}
\begin{abstract}
Nous proposons dans ce travail de minorer de manière optimale la fonction complexité associée à une translation du plan,  qui induit une rotation ergodique du tore $\mathbb R^2/\mathbb Z^2$. De plus, nous donnons un calcul explicite de cette complexité pour l'application $(x,y)\mapsto(x+1/\phi^2,y+x-1/(2\phi^3))$, où $\phi$ désigne le nombre d'or.
\end{abstract}

\section{Introduction} \label{se:intro}

Nous poursuivons dans ce travail, l'étude des interactions entre les systèmes dynamiques et les suites infinies avec un nombre fini de symboles. Nous nous concentrons principalement sur deux types de systèmes : les \textit{rotations} $(x,y)\mapsto(x+\alpha,y+\beta)$ du tore $\mathbb R^2/\mathbb Z^2$, et les \textit{nilsystèmes affines}, qui sont des extensions de rotations du type $(x,y)\mapsto(x+\alpha,y+x+\beta)$. Ces systèmes sont particulièrement intéressants car ils interviennent dans des problèmes délicats d'arithmétique, que ce soit le problème de généralisation de l'algorithme des fractions continues à la dimension 2 pour les rotations du tore, ou celui de la répartition des suites $( \alpha n^2$ mod $1)_n$ pour les nilsystèmes affines.\\

Il est alors relativement naturel de vouloir conjuguer ces systèmes avec des échanges de morceaux du plan. L'existence d'une telle conjugaison est une question délicate, liée à la recherche de partitions génératrices, sur laquelle nous ne nous attarderons pas : il existe des échanges de morceaux conjugués aux translations \cite{MR802634}, et aux nilsystèmes affines \cite{MR1390569}. Les échanges de morceaux se font dans un cas par translation et dans l'autre de manière affine.\\

Pour chaque système, il peut exister de nombreux échanges de morceaux qui lui sont conjugués. Nous pouvons associer à chaque échange de morceaux une fonction \textit{complexité} qui compte le nombre de morceaux dans le raffinement de la partition formée par les morceaux initiaux. Entre deux échanges de morceaux, nous privilégierons celui qui a la fonction complexité la plus petite à celui qui a les morceaux les plus réguliers. Empiriquement, son étude donne plus de résultats sur le comportement du système.\\

Par exemple, pour tout couple $(\alpha,\beta)$ rationnellement indépendants, P. Arnoux, C. Mauduit, L. Shiokawa et J.-I. Tamura étudient dans \cite{MR1279582,MR1259106} des échanges de morceaux polygonaux par translation, conjugués aux rotations du tore $(x,y)\mapsto(x+\alpha,y+\beta)$. Ils arrivent à calculer explicitement la fonction complexité conjecturée par G. Rauzy dans \cite{MR802634}, et trouvent pour tout entier $n$ : $p(n)=n^2+n+1$. 

Cependant, pour certains couples $(\alpha,\beta)$, nous pouvons trouver des échanges de morceaux de complexité $p(n)=2n+1$, dont le plus célèbre est celui qui échange les pièces du fractal de Rauzy. Bien que les morceaux soient à bords fractals, leurs échanges nous donnent beaucoup plus d'informations sur le système. Nous renvoyons à \cite{MR2180244} pour mieux comprendre les interactions entre l'étude des suites infinies avec un nombre fini de symboles et les translations du cercle ou du tore.\\

Nous commençons par vérifier que cette complexité en $2n+1$ est en fait optimale, puisque nous démontrons que tout échange de morceaux conjugué à une translation du plan est de complexité au moins $n\to2n+1$. Cependant, dans une deuxième partie, nous verrons que les nilsystèmes affines se comportent différemment. En effet, pour un exemple bien choisi, la complexité d'un de ces systèmes est égale à la complexité de la rotation du cercle dont il est l'extension : $n\to n+1$.

\section{Résultats principaux et notations} \label{se:res}

Soit $T$ une application mesurable du plan $\mathbb R^2$ dans lui-même qui préserve la mesure de Lebesgue du plan $\mu$ et qui est $\mu$-presque sûrement $\mathbb Z^2$-périodique. C'est-à-dire
\[
\mu\big{(}\{\boldsymbol{x}\in \mathbb R^2; T(\boldsymbol{x}+\mathbb Z^2)\not\subset \{T(\boldsymbol{x})\}+\mathbb Z^2\}\big{)}=0.
\]

De telles applications induisent des applications $\overline{T}$ du tore $\mathbb R^2/\mathbb Z^2$ dans lui-même. Nous noterons $\overline{\mu}$ la mesure de Lebesgue du tore et  $\overline{\mathcal{T}}$ le système dynamique $(\overline{T},\mathbb R^2/\mathbb Z^2,\overline{\mu})$.\\

Un échange de $m$ morceaux du plan $\mathcal{R}=(R,D,\mu)$ est dit \textit{adapté} à l'application $T$ s'il vérifie les propriétés suivantes.
\begin{enumerate}[1)]
\item Le système dynamique qu'il engendre est équivalent en mesure au système $\overline{\mathcal{T}}$.
\item Le sous-ensemble du plan $D$ est la réunion presque sûrement disjointe de $m$ morceaux mesurables $(D_n)_{n\in \{1,\dots,m\}}$ (i.e. $\forall (i,j) \in \{1,\dots,m\}^2$, $\mu(D_i\cap D_j)=0$).
\item La mesure de l'ensemble des points $\boldsymbol{x}$ de $D$ tels qu'il existe un point $\boldsymbol{y}$ de $D$ distinct de $\boldsymbol{x}$, vérifiant $\boldsymbol{x} -\boldsymbol{y} \in \mathbb Z^2$, est nulle.
\item Pour tout point $\boldsymbol{x}$ de $D$, $R(\boldsymbol{x})=T(\boldsymbol{x})-\boldsymbol{n}_i$ si $\boldsymbol{x}\in D_i$, avec  $\boldsymbol{n}_i=(n_i,m_i)\in \mathbb Z^2$.
\end{enumerate}
En particulier, avec cette définition, nous n'imposons pas aux morceaux  d'être bornés.\\

Considérons par exemple l'application $T:(x,y)\mapsto(2x,3y)$. L'échange des $6$ morceaux $D_{i,j} = [i/2,(i+1)/2[\times [j/3,(j+1)/3[$ pour tout $(i,j)\in\{0,1\}\times \{0,1,2\}$, sur lesquels agit l'application $R(\boldsymbol{x})=T(\boldsymbol{x})-(i,j)$, si $\boldsymbol{x}\in D_{i,j}$, est adapté à $T$.\\

Pour chaque échange de morceaux $\mathcal{R}$ adapté à $T$, le codage de l'orbite des points nous permet de définir un langage $\mathfrak  L_{\mathcal{R}}$ de $\{1,\dots,m\}^*$ ; ainsi qu'un système dynamique symbolique $\mathcal{S}_{\mathcal{R}}$. Nous noterons $p_{\mathcal{R}}$ la fonction complexité du langage $\mathcal L_{\mathcal{R}}$. Lorsque le système dynamique symbolique $\mathcal{S}_{\mathcal{R}}$ est conjugué en mesure au système dynamique $\overline{\mathcal{T}}$, nous dirons que l'échange de morceaux $\mathcal{R}$ est \textit{conjugué} à ${T}$. Nous reviendrons plus longuement sur ces différents objets dans la section \ref{se:os}.\\

Lorsqu'il existe un échange de morceaux conjugué à $T$, nous pouvons définir la fonction complexité associée à l'application $T$ pour tout entier $n$, par 
\[
p^{T}  (n) = \inf \Big{\{} p_{\mathcal{R}}(n)  \mbox{, où $\mathcal{R}$ est un échange de morceaux conjugué à $T$} \Big{\}}.
\]
Par le théorème de Krieger, lorsque le système $\overline{\mathcal{T}}$ est d'entropie nulle pour la mesure $\overline{\mu}$, la fonction complexité associée à l'application $T$ est bien définie.\\

Dans la section \ref{se:preuve1}, nous minorons la fonction complexité associée à une translation du plan.
\begin{theoreme} \label{th:1}
Si $T$ est une translation du plan $\mathbb R^2$, telle que $\overline{\mathcal{T}}$ soit ergodique, alors la fonction complexité associée à $T$ est bien définie et pour tout entier $n$, elle est supérieure à $2n+1$.
\end{theoreme}

Dans \cite{MR802634}, G. Rauzy  conjugue les rotations du tore avec des échanges de $3$ morceaux du plan. Donc, pour une translation $T$ du plan, qui induit une translation ergodique du tore, la première valeur de la fonction complexité $p^T(1)$ est $3$. D'autre part, la translation $T$ conjuguée à l'échange des pièces du fractal de Rauzy est de complexité $n \to 2n+1$. À l'heure actuelle, nous ne disposons que d'arguments constructifs pour majorer la complexité de toutes les translations du plan. Savoir si une translation du plan ergodique sur le tore est de complexité $n\to 2n+1$ et si la complexité est atteinte pour tous les entiers par le même échange de morceaux, est une question encore ouverte, et à laquelle il semble très difficile d'apporter une réponse.\\

Dans la section \ref{se:preuve2},  nous construirons une suite d'échanges de morceaux,  conjugués à un nilsystème affine bien choisi, dont la complexité décroît. Nous reprenons des idées utilisées par G. Rauzy, pour construire ce qui est appelé aujourd'hui  le fractal de Rauzy. Mais contrairement à sa construction, la nôtre ne converge pas vers un «échange de morceaux bornés limite» (cf. remarque \ref{rema}). Une des conséquences de ce résultat est qu'il n'y a donc pas, \textit{a priori}, d'échanges de morceaux bornés privilégiés du point de vue de la complexité pour  s'intéresser aux nilsystèmes affines.

\begin{theoreme} \label{th:2}
Notons $T_\phi$ l'application définie par 
\begin{equation}
T_\phi(x,y) =\big{(} x+1/\phi^2,y+x-1/(2\phi^3) \big{)} \mbox{, où $\phi$ est le nombre d'or. }
\end{equation}
Pour tout entier $M$ fixé, il existe un échange de morceaux $\mathcal R_M$ conjugué à $\overline{\mathcal T}_\phi$ tel que pour tout entier $1\leq k \leq M$, $p_{\mathcal R_M} (k)=k+1$. C'est-à-dire, pour tout entier $n$, $p^{T_\phi} (n)= n+1$.
\end{theoreme}

\section{Objets symboliques} \label{se:os}

Nous reprenons les notations de l'introduction. Soit $T$ une application mesurable du plan $\mu$-presque sûrement $\mathbb Z^2$-périodique et qui préserve la mesure de Lebesgue. Fixons $\mathcal{R}$ un échange de $m$ morceaux conjugué à $T$. \\

Pour $\mu$-presque tout point $\boldsymbol{x}\in D$, $C(\boldsymbol{x})$ désigne la suite d'éléments de $\{1,\dots,m\}^\mathbb N$ dont le $n$-ième terme est $i$ si $R^n(\boldsymbol{x})\in D_i$. Il est relativement clair que cette application de \textit{codage} $C$ est mesurable. Nous noterons $S_{\mathcal{R}}$ l'ensemble des suites obtenues de cette manière et $\mu_{\mathcal{R}}$ la mesure de probabilité sur $S_{\mathcal{R}}$ qui provient de la mesure de Lebesgue $\mu$ via $C$ (c'est-à-dire pour tout ensemble mesurable $A$, $\mu_{\mathcal{R}}(A)=\mu(C^{-1}A)$). Le \textit{système dynamique symbolique} $\mathcal{S}_{\mathcal{R}}$ associé à ${\mathcal{R}}$ est le triplet $(\mathcal S,S_R,\mu_{\mathcal{R}})$, où $\mathcal S$ désigne le décalage classique. De cette manière, il est clair que le système $\mathcal{S}_{\mathcal{R}}$ est un facteur mesuré du système $\overline{\mathcal{T}}$ et que l'application de codage $C$ est l'application facteur associée. Les codages ainsi obtenus sont les codages «naturels» décrits par V. Berthé, S. Ferenczi et L.Q. Zamboni dans \cite{MR2180244}.\\

Notons $\{1,\dots,m\}^*$ l'ensemble des mots à valeur dans $\{1,\dots,m\}$ finis ou infinis. La longueur d'un mot $u\in\{1,\dots,m\}^*$ sera notée $| u | \in \mathbb N \cup \{+\infty\}$. Soient $v \in  \{1,\dots,m\}^{\mathbb N}$ et $u \in \{1,\dots,m\}^*$. Le mot $u$ est un \textit{facteur} de $v$ s'il existe un entier $m$ tel que pour tout entier $n\leq | u|$, $u_n=v_{m+n}$. Un \textit{langage} est un sous-ensemble de $\{0,1\}^*$. Il est dit \textit{factoriel} si pour tout mot $v$ du langage, tout facteur $u$ de $v$ appartient au langage. Nous noterons $\mathcal L(\mathfrak L)$ le langage factoriel engendré par  un langage $\mathfrak L$. Nous définissons le langage $ \mathfrak L_{\mathcal{R}} $ associé à l'échange de morceaux $\mathcal{R}$, par  $ \mathfrak L_{\mathcal{R}} = \mathcal L(S_{\mathcal{R}})$. La fonction complexité $p_{\mathcal R}$ associée au langage $\mathfrak L_{\mathcal{R}}$ est définie pour tout entier $n$ par :
\[
p_{\mathcal{R}} (n) = \# \{ u \in \mathfrak L_{\mathcal{R}} \mbox{ tel que } | u |=n\}.
\]

Ces définitions ne sont sûrement pas les meilleures pour une étude générale des applications du tore, mais elles suffiront amplement pour s'intéresser aux classes de systèmes très particulières que sont les translations et les nilsystèmes affines.\\

Nous rappelons également le résultat suivant de P. Halmos  qui permet de déterminer si le système symbolique associé à un échange de morceaux adapté à $T$ est conjugué en mesure au système $\overline{\mathcal{T}}$.

\begin{proposition}[P. Halmos] \label{prop:raff}
Soit $T$ une application du plan $\mu$-presque sûrement $\mathbb Z^2$-périodique, qui préserve la mesure de Lebesgue et telle que le système $\overline{\mathcal{T}}$ soit ergodique. Soit $\mathcal{R}$ un échange de $m$ morceaux du plan adapté à $T$. Alors, les assertions suivantes sont équivalentes.
\begin{enumerate}[1)]
\item L'échange de morceaux $\mathcal{R}$  est conjugué à $T$.
\item Pour toute suite $(u_n)\in \{1,\dots,m\}^{\mathbb N}$,  \[
\mu\left( \bigcap \limits_{k=1}^N R^{-k} D_{u_k} \right) \underset{N\to+\infty}{\longrightarrow } 0.
\]
\item La mesure de Lebesgue de l'ensemble $\{\boldsymbol{x}\in D$ tel qu'il existe  $\boldsymbol{y}\in D\setminus\{\boldsymbol{x}\}$ tel que $C(\boldsymbol{x})=C(\boldsymbol{y})\}$ est nulle.
\end{enumerate}
\end{proposition}

\begin{proposition} \label{prop:n+1}
Soit $\overline{\mathcal{T}}$ un système dynamique défini sur le tore $\mathbb R^2/\mathbb Z^2$. Fixons $T$ un relèvement de $\overline{T}$ défini dans le plan et supposons qu'il existe un échange de morceaux conjugué à $T$. Nous supposons, de plus, que le système $\overline{\mathcal{T}}$  est totalement ergodique (c'est-à-dire que toutes ses puissances $\overline{\mathcal{T}}^k=(\overline{T}^k,\mathbb R^2/\mathbb Z^2,\overline{\mu})$ sont ergodiques). Alors, la fonction complexité associée à $T$ existe et croît strictement.
En particulier, pour tout entier $n$ elle est supérieure ou égale à $n+1$. 
\end{proposition}

\begin{proof}[de la proposition \ref{prop:n+1}]
Fixons un système $\overline{\mathcal{T}}$ satisfaisant les hypothèses de la proposition. Supposons alors qu'un des échanges de morceaux conjugués à $T$ ait une $n$-ième complexité égale à $j$ et une $(n+1)$-ième complexité encore égale à $j$. Cela veut donc dire que la partition initiale se raffine en $j$ morceaux $P_1$, $P_2$, $\dots$ $P_j$, lorsque l'on s'interesse à ses $(n-1)$-ièmes pré-images. Puisque la $(n+1)$-ième complexité est encore égale à $j$, cela signifie que la partition $P_1$, $\dots$ $P_j$ ne se raffine pas. Par ergodicité, l'application $\boldsymbol{x}\mapsto i\in\{1,\dots,j\}$ si $\boldsymbol{x}\in P_i$ définit un facteur mesuré. Or, puisque le système est totalement ergodique, il n'admet pas de facteur fini. La fonction $p^{T}$ est donc strictement croissante.

On conclut alors en remarquant qu'il faut échanger au moins deux morceaux pour conjuguer le système ergodique du tore $\overline{\mathcal{T}}$ avec le système symbolique engendré par l'échange de morceaux.
\end{proof}

\section{Preuve du théorème \ref{th:1}}
\label{se:preuve1}

\begin{lemme} \label{le:1}
Si $T$ est une translation du plan telle que $\overline{\mathcal{T}}$ soit ergodique, alors $p^{T} (1)\geq 3$.
\end{lemme}

\begin{proof}[Preuve du lemme \ref{le:1}]
Fixons deux réels $r$ et $s$ tels que la translation $T:(x,y) \mapsto (x+r,y+s)$ soit ergodique sur le tore $\mathbb R^2/\mathbb Z^2$ (c'est-à-dire pour tout couple d'entiers $(n,m)$ non-nuls simultanément, $nr+ms \notin \mathbb Z$).\\

Supposons que l'application $T$ soit  conjuguée à un échange de deux morceaux $D_1$ et $D_2$, de mesure strictement positive (respectivement $\alpha_1$ et $\alpha_2$). Puisque $\mu(D_1\cup D_2)=1$, et que $\mu(D_1\cap D_2)=0$, on a immédiatement $\alpha_1+\alpha_2=1$. Par le théorème ergodique de Birkhoff et de récurrence de Poincaré, il existe un point récurent $\boldsymbol{x}_0=(x_0,y_0)$ de $D$, tel que 
\[
\frac{1}{N} \sum \limits_{n=0} ^{N-1} \mathds{1}_{R^{-n} D_1} (\boldsymbol{x}_0) \to \alpha_1 \mbox{ et } \frac{1}{N} \sum \limits_{n=0} ^{N-1} \mathds{1}_{R^{-n} D_2} (\boldsymbol{x}_0) \to \alpha_2.
\]

Traduisons maintenant la condition de récurrence imposée sur le point $\boldsymbol{x}_0$. Pour tout entier $N$ :
\[
R^N (\boldsymbol{x}_0) = 
\begin{pmatrix}
x_0+Nr- n_1\sum \limits_{n=0} ^{N-1} \mathds{1}_{R^{-n} D_1} (\boldsymbol{x}_0) -n_2 \sum \limits_{n=0} ^{N-1} \mathds{1}_{R^{-n} D_2} (\boldsymbol{x}_0)  \\
y_0+Ns- m_1\sum \limits_{n=0} ^{N-1} \mathds{1}_{R^{-n} D_1} (\boldsymbol{x}_0) -m_2 \sum \limits_{n=0} ^{N-1} \mathds{1}_{R^{-n} D_2} (\boldsymbol{x}_0) 
\end{pmatrix}.
\]

Puisque le point $\boldsymbol{x}_0$ est récurent, il existe une sous-suite croissante d'entiers $(N_k)_k$ telle que $R^{N_k} (\boldsymbol{x}_0) /N_k$ converge vers $\boldsymbol{0}$. Ce qui nous amène aux équations :
\[
r = n_1\alpha_1 + n_2 \alpha_2 \mbox{ et }s = m_1\alpha_1 + m_b \alpha_2.
\]
Nous trouvons donc : 
\[
r = n_1 +\alpha_2 ( n_2-n_1) \mbox{ et } s = m_1+\alpha_2 (m_2-m_1).
\]
En remarquant que les quantités $n_2-n_1$ et $m_2-m_1$ sont non nulles (sinon $r$ ou $s$ serait entier), nous trouvons alors que $(m_2-m_1)r-(n_2-n_1)s\in \mathbb Z$, ce qui est absurde.
\end{proof}

\begin{proposition} \label{pr:1}
Si $T$ est une translation du plan telle que $\overline{\mathcal{T}}$ soit ergodique, alors $p^{T} (2)\geq 5$.
\end{proposition}

\begin{proof}[Preuve de la proposition \ref{pr:1}]
Fixons deux réels $r$ et $s$ tels que la translation $T:(x,y) \mapsto (x+r,y+s)$ soit ergodique sur le tore $\mathbb R^2/\mathbb Z^2$, et $\mathcal{R}$ un échange de trois morceaux conjugué à $T$, tel que $p_{\mathcal{R}}(2)\leq 4$. \\

Par la proposition \ref{prop:n+1}, la fonction complexité  $p_{\mathcal{R}}$ est une fonction strictement croissante. De plus, nous avons vu dans le lemme \ref{le:1} que  $p_{\mathcal{R}}(1)=3$. Il nous faut donc montrer qu'il est impossible d'avoir un échange de morceaux $\mathcal{R}$ conjugué à $T$ tel que $p_{\mathcal{R}}(1)=3$ et $p_{\mathcal{R}}(2)=4$.\\

Supposons qu'il existe $\mathcal{R}$, un échange de trois morceaux $D_1$, $D_2$ et $D_3$ de mesures respectives $\alpha_1$, $\alpha_2$ et $\alpha_3$, adapté à ${T}$, tel que $p_{\mathcal{R}}(2)=4$. Nécessairement, un seul des trois morceaux se raffine quand on s'interesse à la pré-image de la partition formée par les trois morceaux $D_1$, $D_2$ et $D_3$. Supposons donc que le morceau $D_1$ se raffine en deux morceaux $D_1^{(1)}$ et $D_1^{(2)}$, de mesure $\alpha_1^{(1)}$ et $\alpha_1^{(2)}$. En reprenant les arguments de la preuve du lemme \ref{le:1}, nous pouvons montrer qu'il existe six entiers $n_1$, $n_2$, $n_3$, $m_1$, $m_2$ et $m_3$, tels que :
\begin{equation} \label{eq:pour2}
\left\{ \begin{array}{ccccc}
& & 1  & =&\alpha_1^{(1)} +\alpha_1^{(2)}+\alpha_2+\alpha_3 ,\\
&&r &=& \alpha_1^{(1)} n_1+\alpha_1^{(2)} n_1+\alpha_2 n_2+  \alpha_3 n_3,\\
&\mbox{et}&  s &=& \alpha_1^{(1)} m_1+\alpha_1^{(2)} m_1+\alpha_2 m_2  +\alpha m_3.
\end{array} \right.
\end{equation}

Puisque les morceaux $D_2$ et $D_3$ ne se raffinent pas, il existe $(i,j)\in\{1,2,3\}^2$ tels que $R(D_2)\subset D_i$ et $R(D_3)\subset D_j$. Remarquons immédiatement que par ergodicité, et puisque $R$ préserve la mesure, il est impossible que $R(D_2)\subset D_2$, ou que $R(D_3)\subset D_3$. Quitte à renuméroter $D_1^{(1)}$ et $D_1^{(2)}$, ou $2$ et $3$, il n'y a que trois cas possibles :
\[
\renewcommand{\arraystretch}{1.4} 
\begin{array}{|c||c|c|c|c|c|}
\hline
 \mbox{Cas 1} &   R(D_1^{(2)}) \subset D_3  & R(D_1^{(1)}) \subset D_2     & R(D_2)     \subset D_1     &  R(D_3)     \subset D_1    \\
 \hline
 \mbox{Cas 2} & R(D_1^{(2)}) \subset  D_3      & R(D_1^{(1)}) \subset  D_1     & R(D_2) \subset D_1           &R(D_3) \subset D_2        \\
 \hline
\mbox{Cas 3}  &  R(D_1^{(2)}) \subset D_3 &  R(D_1^{(1)}) \subset D_2 & R(D_2)  \subset  D_1  & R(D_3) \subset D_2 \\
\hline
\end{array}
\]

Nous allons étudier chacun des cas et vérifier qu'ils mènent tous à une absurdité. Pour cela, nous allons traduire les répercussions de ces inclusions sur les valeurs de $\alpha_1^1$, $\alpha_1^2$, $\alpha_2$ et $\alpha_3$ ; en utilisant le fait que l'application $R$ préserve la mesure de Lebesgue. \\

\textbf{Premier cas.}
Puisque $R(D_3)\cup R(D_2) \subset D_1$, alors $\mu(R(D_3)\cup R(D_2) )\leq \alpha_1=\alpha_1^{(1)}+\alpha_1^{(2)}$. De plus, $\mu(R(D_3)\cup R(D_2) )\leq \alpha_2+\alpha_3$. De même, $\alpha_1^{(1)}\leq \alpha_2$ et $\alpha_1^{(2)}\leq \alpha_3$. Puisque $\alpha_1^{(1)}+\alpha_1^{(2)}+\alpha_2+\alpha_3=1$,  alors nécessairement :
\[
\alpha_2+\alpha_3=\alpha_1^{(1)}+\alpha_1^{(2)} , \ \alpha_1^{(1)}= \alpha_2 \mbox{ et } \alpha_1^{(2)}= \alpha_3. 
\]
Le système \ref{eq:pour2} devient alors :
\[
1/2   = \alpha_2+\alpha_3,\ r = \alpha_2 (n_1+n_b)+  \alpha_3 (n_1+n_3)  \mbox{ et } s  = \alpha_2(m_1+ m_2)  +\alpha_3 (m_1+m_3).
\]
Comme nous l'avons vu dans la preuve du lemme \ref{le:1}, les réels $r$ et $s$ sont alors rationnellement dépendants et le système $\overline{\mathcal{T}}$ n'est donc pas ergodique. C'est absurde. \\

\textbf{Deuxième cas.}
L'étude des inclusions nous amène aux conditions suivantes :  
\[
\alpha_3=\alpha_2, \ \alpha_2+\alpha_1^{(1)}=\alpha_1=\alpha_1^{(1)}+\alpha_1^{(2)} \mbox{ et } \alpha_1^{(2)}=\alpha_3. 
\]
Donc $\alpha_3=\alpha_2=\alpha_1^{(2)}$ et le système \ref{eq:pour2} devient alors :
\[
1   = \alpha_1^{(1)}+ 3\alpha_3 ,\ r = \alpha_1^{(1)} n_1 +  \alpha_3 (n_1+n_2+n_3)  \mbox{ et }  s = \alpha_1^{(1)}m_1  +\alpha_3 (m_1+m_2+m_3).
\]
Comme précédemment, cela nous amène à une absurdité. \\

\textbf{Troisième cas.}
L'étude des inclusions impose les conditions suivantes :  
\[
\alpha_3+\alpha_1^{(1)}=\alpha_2 , \ \alpha_2=\alpha_1=\alpha_1^{(1)}+\alpha_1^{(2)} \mbox{ et } \alpha_1^{(2)}=\alpha_3. 
\]
Donc $\alpha_3=\alpha_2=\alpha_1^{(2)}$ et le système \ref{eq:pour2} s'écrit :
\[
1  = 2 \alpha_2+ \alpha_3 , \ r = \alpha_2(n_1+n_2) +  \alpha_3 n_3  \mbox{ et } s = \alpha_2(m_1+m_2) +  \alpha_3 m_3 .
\]
Encore comme précédemment, cela nous amène à une contradiction.
\end{proof}

\begin{proof}[Preuve du théorème \ref{th:1}]
Fixons deux réels $r$ et $s$ tels que la translation $T:(x,y) \mapsto (x+r,y+s)$ soit ergodique sur le tore $\mathbb R^2/\mathbb Z^2$, ainsi qu'un échange de morceaux $\mathcal{R}$ conjugué à $T$. Fixons également un entier $n$ jusqu'auquel le résultat du théorème est vrai et montrons que le résultat persiste jusqu'au cran $n+1$.  

Comme dans la preuve de la proposition \ref{pr:1}, il y a des situations trivialement dégénérées. Supposons par la suite que pour tout entier $i\in\{1,\dots,n\}$, $p_{\mathcal{R}}(i)= 2i+1$, et montrons qu'il est impossible que $p_{\mathcal{R}}(n+1)$ soit égale à $2n+2$.

Le $(n-1)$-ième raffinement des morceaux initiaux est donc composé de $2n+1$ morceaux notés $(D_1, \ldots, D_{2n+1})$ de mesure $\alpha_1,\dots,\alpha_{2n+1}$. Un seul de ces $2n+1$ morceaux se raffine quand on s'interesse à la $n$-ième pré-image de la partition formée par les trois morceaux initiaux. Supposons donc que le morceau $D_1$ se raffine en deux morceaux $D_1^{(1)}$ et $D_1^{(2)}$, de mesure $\alpha_1^{(1)}$ et $\alpha_1^{(2)}$. Comme précédemment, pour tout $i\in\{2,\dots,2n+1\}$, il existe $j\in\{1,\dots,2n+1\}$ tel que $R(D_i)\subset D_j$.

Nous pouvons utiliser les arguments du lemme \ref{le:1} pour monter qu'il existe des entiers $(n_i)_{i=1\dots2n+1}$ et $(m_i)_{i=1\dots2n+1}$ tels que :
\begin{equation} \label{eq:pourn}
\left\{ \begin{array}{ccccll}
& & 1  & =&\alpha_1^{(1)} + \alpha_1^{(2)}+\alpha_2+\alpha_3+\dots+\alpha_{2n+1} ,\\
&&r &=& \alpha_1^{(1)} n_1+\alpha_1^{(2)} n_1+\alpha_2 n_2+ \dots+ \alpha_{2n+1}n_{2n+1} ,\\
&\mbox{et}&  s &=& \alpha_1^{(1)} m_1+\alpha_1^{(2)} m_1+\alpha_2 m_2+ \dots+ \alpha_{2n+1}m_{2n+1} .
\end{array} \right.
\end{equation}

Quitte à réindexer les indices , deux cas peuvent alors se produire :
\[
\begin{array}{clcl}
&\mbox{ Cas 1 : }&  R\big{(}D_1^{(1)}\big{)} \subset D_1 \mbox{ et } R\big{(}D_1^{(2)}\big{)} \subset D_2. \\
&\mbox{ Cas 2 : }&  R\big{(}D_1^{(1)}\big{)} \subset D_2 \mbox{ et } R\big{(}D_1^{(2)}\big{)} \subset D_3.
\end{array}
\]

\textbf{Premier cas.} C'est le cas le plus simple : une seule zone peut rentrer dans la zone $D_1$. Quitte à réindexer les indices, on peut voir que puisque le système $\mathcal{R}$ est ergodique, il n'est pas possible que la réunion de certaines zones soit envoyée dans elle-même. Donc $R(D_2)\subset D_3$, $R(D_3)\subset D_4$, $\dots$, $R(D_{2n}) \subset D_{2n+1}$ et $R(D_{2n+1}) \subset D_{1}$. Ce qui nous amène aux contraintes suivantes sur les coefficients $\alpha_i$ : 
\[
\alpha_1^{(1)}+\alpha_{2n+1}=\alpha_1^{(1)}+\alpha_1^{(2)}, \ \alpha_1^{(2)}= \alpha_2=\alpha_3, \ \dots \dots, \alpha_{2n-1}=\alpha_{2n} \mbox { et } \alpha_{2n}=\alpha_{2n+1}.
\]
Soit encore : $\alpha_1^{(2)}=\alpha_2=\alpha_3=\dots=\alpha_{2n+1}$. Le système \ref{eq:pourn} se réécrit donc :
\[
1  = \alpha_1^{(1)} +(2n+1) \alpha_1^{(2)}, \ r = \alpha_1^{(1)} n_1+\alpha_1^2 N \mbox{ et } s = \alpha_1^{(1)} m_1+\alpha_1^2 M,
\]
où $N = \sum \limits_{i=1}^{2n+1} n_i $ et $M = \sum \limits_{i=1}^{2n+1} m_i$.
On conclut alors comme la preuve de la proposition \ref{pr:1}. \\

\textbf{Deuxième cas.} 
C'est plus compliqué. En effet, deux zones peuvent rentrer dans la zone $D_1$. Comme dans la preuve de la proposition \ref{pr:1}, quitte à réindexer les indices, deux sous-cas peuvent se produire.\\

\begin{tikzpicture}[scale=0.8]
\node at(-2.8,0.5) {\underline{Premier sous-cas} : };
 \node [black] (B) at (0,0.5) {$D_1^{(1)}$};  \node [black] (C) at (1.4,0.5) {$D_2$};       \node [black] (D) at (2.7,0.5) {$D_4$}; \node [black] (E) at (4.5,0.5) {$D_{2k}$}; 
\node [black] (BB) at (0,-0.5) {$D_1^{(2)}$};   \node [black] (CC) at (1.4,-0.5) {$D_3$};  \node [black] (DD) at (2.7,-0.5) {$D_5$};  \node [black] (EE) at (4.5,-0.5) {$D_{2k+1}$};   
\node [black] (F) at (6.5,0) {$D_{2k+2}$};  \node [black] (G) at (9,0) {$D_{2n+1}$};  \node [black] (H) at (10.5,0) {$D_{1}.$}; 
\draw[black,->] (B) -- (C); \draw[black,->] (C) --    (D); \draw[black,densely dashed,->] (D) --    (E); \draw[black,->] (E) --    (F); \draw[black,->] (BB) --    (CC); \draw[black,->] (CC) --    (DD);
\draw[black,densely dashed,->] (DD) --    (EE); \draw[black,->] (EE) --    (F); \draw[black,densely dashed,->] (F) --    (G); \draw[black,->] (G) --    (H);
\end{tikzpicture} 
\mbox{} \\

\begin{center} \begin{tikzpicture}[scale=0.8]
\node at (-2.8,0.5) {\underline{Deuxième sous-cas} :};
\node [black] (B) at (0,0.5) {$D_1^{(1)}$};  \node [black] (C) at (1.4,0.5) {$D_2$};  \node [black] (D) at (2.7,0.5) {$D_4$}; \node [black] (E) at (4.5,0.5) {$D_{2k}$};  \node [black] (F) at (6.5,0.5) {$D_{1},$};
 \node [black] (BB) at (0,-0.5) {$D_1^{(2)}$};  \node [black] (CC) at (1.4,-0.5) {$D_3$};  \node [black] (DD) at (2.7,-0.5) {$D_5$}; \node [black] (EE) at (4.5,-0.5) {$D_{2k+1}$}; 
\node [black] (FF) at (6.5,-0.5) {$D_{2k+2}$};  \node [black] (G) at (8.8,-0.5) {$D_{2n+1}$};  \node [black] (H) at (10.5,-0.5) {$D_{1}.$}; 
\draw[black,->] (B) --    (C); \draw[black,->] (C) --    (D); \draw[black,densely dashed,->] (D) --    (E); \draw[black,->] (BB) --    (CC); \draw[black,->] (CC) --    (DD); \draw[black,densely dashed,->] (DD) --    (EE);
\draw[black,->] (EE) --    (FF); \draw[black,->] (E) --    (F); \draw[black,densely dashed,->] (FF) --    (G); \draw[black,->] (G) --    (H);
\end{tikzpicture} \end{center}

\indent
\textbf{Premier sous-cas.} Les conditions sur les coefficients sont :
\[
\left\{ \begin{array}{cccccccccccc} \alpha_{2k}& =&\alpha_{2k-2}& =&\dots&=&\alpha_1^{(1)} , \\ \alpha_{2k+1} &=& \alpha_{2k-1}& =&\dots&=&\alpha_1^{(2)} , \\  \end{array} \right.
\]
et $\alpha_{2k}+\alpha_{2k+1} = \alpha_{2k+2}  = \dots = \alpha_{2n+1}=\alpha_1=\alpha_1^{(1)}+\alpha_1^{(2)}$.\\ 
Posons  $N_1=n_1+ \sum \limits _{i=1}^{k} n_{2i} +\sum \limits _{i=2k+2}^{2n+1} n_i$, $N_2=n_1+ \sum \limits _{i=1}^{k} n_{2i+1} +\sum \limits _{i=2k+2}^{2n+1} n_i$,  $M_1=m_1+ \sum \limits _{i=1}^{k} m_{2i} +\sum \limits _{i=2k+2}^{2n+1} m_i$,  et $M_2=m_1+ \sum \limits _{i=1}^{k} m_{2i+1} +\sum \limits _{i=2k+2}^{2n+1} m_i$. Le système \ref{eq:pourn} devient donc :
\[
\left\{ \begin{array}{ccccc} & & 1  & =& (2n-k+1) \alpha_1^{(1)} + (2n-k+1) \alpha_1^{(2)}   ,\\ &&r &=& \alpha_1^{(1)} N_1 +\alpha_1^{(2)} N_2,\\ &\mbox{et}&  s &=& \alpha_1^{(1)} M_1+\alpha_1^{(2)} M_2 . \end{array} \right.
\]
On conclut alors comme précédemment.\\

\textbf{Deuxième sous-cas.} Les conditions sur les coefficients sont :
\[
\left\{ \begin{array}{ll} 
\alpha_{2k}  = \alpha_{2k-2}  = \dots = \alpha_2 =\alpha_1^{(1)}, \\ \alpha_{2n+1}  =\alpha_{2n}=\dots=\alpha_{2k+1} =\dots =\alpha_3=\alpha_1^{(2)}, \\ \alpha_{2k}+\alpha_{2n+1}=\alpha_1=\alpha_1^{(1)}+\alpha_1^{(2)}.
\end{array} \right.
\]
En posant alors $N_1=n_1+ \sum \limits _{i=1}^{k} n_{2i}$, $N_2=n_1+ \sum \limits _{i=1}^{k} n_{2i+1} +\sum \limits _{i=2k+2}^{2n+1} n_i$, $M_1=m_1+ \sum \limits _{i=1}^{k} m_{2i}$, et $M_2=m_1+ \sum \limits _{i=1}^{k} m_{2i+1} +\sum \limits _{i=2k+2}^{2n+1} m_i$, le système \ref{eq:pourn} s'écrit :
\[
\left\{ \begin{array}{cccll} & & 1  & =&\alpha_1^{(1)}(k+1) + \alpha_1^{(2)}(2n-k+1) ,\\ &&r &=& \alpha_1^{(1)} N_1 +\alpha_1^{(2)} N_2,\\ &\mbox{et}&  s &=& \alpha_1^{(1)} M_1+\alpha_1^{(2)} M_2 . \end{array} \right.
\]
On conclut encore comme précédemment.
\end{proof}

\section{Preuve du théorème \ref{th:2}}
\label{se:preuve2}

Notons $\sigma$ la substitution de Fibonacci définie par 
\begin{equation} \label{eq:fibo}
\sigma \ : \ 1 \mapsto 12 \mbox{ et } 2 \mapsto 1.
\end{equation}
La preuve du théorème se déduit des deux résultats suivants. 

\begin{proposition} \label{prop:phi2}
Soit $\mathcal{R}$ un échange de deux morceaux $D_1$ et $D_2$ conjugué à ${T}_\phi$, défini par
\[
 R(\boldsymbol{x})=T_\phi (\boldsymbol{x}) \mbox{ si } \boldsymbol{x}\in D_1 \mbox{, et } R(\boldsymbol{x})=T_\phi(\boldsymbol{x})-(1,0) \mbox{ si } \boldsymbol{x}\in D_2. 
\]
En notant $\mathfrak p$ la projection de $\mathbb R^2$ dans $\mathbb R$ définie par $\mathfrak p(x,y)=x$, nous supposons de plus l'existence d'un réel $z$ tel que $\mathfrak p(D_1)\subset]z-1,z+1/\phi^2]$ et $\mathfrak p(D_2) \subset ]z-1/\phi^2,z+1/\phi^2[$. 

Alors, il existe alors un échange de deux morceaux $\mathcal{R}'$ conjugué à $\mathcal{T}_\phi$, de même nature que $\mathcal{R}$, tel que :
\begin{equation} \label{eq:langa}
\mathfrak L_{\mathcal{R}'} = \mathcal L(\sigma(\mathfrak L_{\mathcal{R}})).
\end{equation}
\end{proposition}

\begin{proposition} \label{prop:phi1}
Il existe $\mathcal{R}^{(1)}$, un échange de deux morceaux bornés, connexes et simplement connexes $D_1^{(1)}$ et $D_2^{(1)}$ conjugué à ${T}_\phi$. Les mesures de $D_1^{(1)}$ et $D_2^{(1)}$ sont respectivement $1/\phi$ et $1/\phi^2$. L'application $R^{(1)}$ est définie par :
\[
R^{(1)}(\boldsymbol{x})=T_\phi(\boldsymbol{x}) \mbox{ si }\boldsymbol{x}\in D_1^{(1)}  \mbox{ , et  }R^{(1)}(\boldsymbol{x})=T_\phi (\boldsymbol{x})-(1,0)  \mbox{ si } \boldsymbol{x}\in D_2^{(1)}.
\]
De plus, il existe un réel $z$ tel que :
\[
\mathfrak p\left( D_1^{(1)}\right)\subset]z-1,z+1/\phi^2] \mbox{ et } \mathfrak p\left(D_2^{(1)} \right) \subset ]z-1/\phi^2,z+1/\phi^2[.
\]
\end{proposition}

\begin{proof}[Preuve de la proposition \ref{prop:phi2}]
Nous allons construire explicitement l'échange de morceaux $\mathcal{R}'$. Pour cela, nous définissons une application $\psi:\mathbb R^2\longrightarrow \mathbb R^2$ par 
\[
\psi(x,y)=(-\phi x, -y+\alpha x^2+\beta x+\gamma),
\]
où $\alpha$, $\beta$ et $\gamma$ sont des paramètres réels que nous fixerons plus tard. L'application $\psi$ est bijective et son inverse est :
\[
\psi ^{-1} (x,y)=( -x/ \phi, -y+\alpha x^2/\phi^2-\beta x/\phi+\gamma). 
\]
Nous posons alors :
\[
D_1 '=\psi^{-1}(D)=\psi^{-1}(D_1\cup D_2), \ D_2' = T_\phi(\psi^{-1}(D_1)) \mbox{ et }D'=D_1'\cup D_2'.
\]
Nous considérons l'application $R'$ définie \textit{a priori} de $D'$ dans $\mathbb R^2$, par 
 \[
 R'(\boldsymbol{x})=T_\phi(\boldsymbol{x}) \mbox{ si } \boldsymbol{x}\in D_1' \mbox{, et } R'(\boldsymbol{x})=T_\phi(\boldsymbol{x})-(1,0) \mbox{ si } \boldsymbol{x}\in D_2' \setminus D_1'.
 \]
 La suite de la preuve est de vérifier que cet échange de morceaux convient pour certaines valeurs des paramètres $\alpha$, $\beta$ et $\gamma$.\\

\textbf{L'application $R'$ est à valeur dans $D'$.} 
Commençons par fixer un élément  $\boldsymbol{x}$ de $\psi^{-1} (D_2)$. Il existe donc $(x,y)\in D_2$ tel que $\boldsymbol{x}= \psi^{-1} (x,y) = (-x/\phi,-y+\alpha x^2/\phi^2-\beta x/\phi+\gamma)$. Nous allons montrer que $T_\phi ( \boldsymbol{x}) \in \psi^{-1} (D)$. 
\[
\begin{array}{clcl}
T_\phi ( \boldsymbol{x})  &=& \begin{pmatrix} -x/\phi +1/\phi^2 \\ -y+\alpha x^2/\phi^2-\beta x/\phi+\gamma -x/\phi-1/(2\phi^3) \end{pmatrix},\\
&=& \begin{pmatrix} -(x +1/\phi^2-1)/\phi \\ -y+\alpha x^2/\phi^2-\beta x/\phi+\gamma -x/\phi-1/(2\phi^3) \end{pmatrix} . \end{array}
\]
Nous allons vérifier qu'il est possible de fixer les paramètres $\alpha$, $\beta$ et $\gamma$, afin que  $T_\phi ( \boldsymbol{x})$ soit égal à $\psi^{-1}(R (x,y ))$. Calculons :
\[
\begin{array}{clll} 
\psi^{-1}(R (x,y )) &=& \psi^{-1} \begin{pmatrix} x +1/\phi^2-1 \\ y+x-1/(2\phi^3) \end{pmatrix} , \\
                               &=& \begin{pmatrix} -(x +1/\phi^2-1)/\phi \\ -y-x+1/(2\phi^3)+\alpha (x-1/\phi)^2/\phi^2-\beta (x-1/\phi)/\phi+\gamma \end{pmatrix} .
\end{array}
\]
Nous voulons donc que l'équation suivante soit satisfaite :
\[
\begin{array}{cc}
-y+\alpha x^2/\phi^2-\beta x/\phi+\gamma -x/\phi-1/(2\phi^3)=-y-x+1/(2\phi^3) +\alpha (x-1/\phi)^2/\phi^2 \\
-\beta (x-1/\phi)/\phi+\gamma.\end{array}
\]
Soit encore : 
\[
\renewcommand{\arraystretch}{1.4} 
\begin{array}{cc}
 \alpha x^2/\phi^2-\beta x/\phi -x/\phi-1/(2\phi^3) = -x+1/(2\phi^3)+\alpha (x-1/\phi)^2/\phi^2-\beta (x-1/\phi)/\phi , \\
 \begin{array}{ccll}
&\Longleftrightarrow & -x/\phi-1/(2\phi^3) = -x+1/(2\phi^3) -2\alpha x/\phi^3+\alpha/\phi^4+\beta /\phi^2, \\
&\Longleftrightarrow &\left\{ \begin{array}{clclc} -1/\phi & = & -1-2\alpha/\phi^3, \\ -1/(2\phi^3)& = &1/(2\phi^3)+\alpha/\phi^4+\beta/\phi^2, \end{array} \right. \\
&\Longleftrightarrow &\left\{ \begin{array}{clclc}  \alpha  & = & -\phi/2, \\ \beta & = & -1/(2\phi). \end{array} \right.
\end{array}
\end{array}
\]
Nous fixons donc pour toute la suite de cette preuve :
\[
\psi(x,y) = (-\phi x, -y-\phi x^2/2+\phi^2 x/2),
\]
 d'inverse $\psi ^{-1} (x,y)=( -x/ \phi, -y- x^2 /( 2\phi)+x/(2\phi^2))$.\\
 
Soit maintenant $\boldsymbol{x}\in D_1'=T_\phi( \psi^{-1} (D_1))$. Il existe donc $(x,y)\in D_1$ tel que 
\[
\boldsymbol{x}= T_\phi (\psi^{-1} (x,y) )= ( -x/ \phi+1/\phi^2,  -y- x^2 /( 2\phi)+x/(2\phi^2)  -x/\phi-1/(2\phi^3)).
\]
Nous allons montrer que $R' ( \boldsymbol{x}) \in \psi^{-1} (D)$. 
\[
R' ( \boldsymbol{x})  = \begin{pmatrix}  -x/ \phi+1/\phi^2 +1/\phi^2-1 \\  -y- x^2 /( 2\phi)+x/(2\phi^2)  -x/\phi-1/(2\phi^3)  -x/\phi+1/\phi^2-1/(2\phi^3) \end{pmatrix}.
\]
Pour cela,  nous vérifions que $R' ( \boldsymbol{x}) =\psi^{-1}(R(x,y))=\psi^{-1}(T_\phi(x,y))$ :
\[
\begin{array}{clll} \psi^{-1}(T_\phi(x,y))&=&\psi^{-1}(x+1/\phi^2,y+x-1/(2\phi^3)), \\&=&\begin{pmatrix}  -x/ \phi-1/\phi^3 \\ -y-x+1/(2\phi^3) - (x+1/\phi^2)^2 / (2\phi)+ (x+1/\phi^2)/(2\phi^2)  \end{pmatrix} \end{array}.
\]
Puisque $-x/ \phi+1/\phi^2 +1/\phi^2-1  = -x/ \phi-1/\phi^3$, il ne nous reste qu'à montrer
que 
\[
\begin{array}{cc}  -y- x^2 /( 2\phi)+x/(2\phi^2)  -x/\phi-1/(2\phi^3)  -x/\phi+1/\phi^2-1/(2\phi^3) =-y-x\\+1/(2\phi^3) - (x+1/\phi^2)^2 / (2\phi)+ (x+1/\phi^2)/(2\phi^2) .\end{array}
\]
Soit encore, 
\[
\begin{array}{l}
- x^2 /( 2\phi)-x/\phi-1/(2\phi^3)  -x/\phi+1/\phi^2-1/(2\phi^3)=-x+1/(2\phi^3) \\
\ \ \ \ \ \ \ \ \ \ \ \ \ \ \ \ \ \ \ \ \ \ \ \ \ \ \ \ \ \ \ \ \ \ \ \ \ \ \ \ \ \ \ \ \ \ \ \ \ \ \ \ \ \ \ \ \ \ \ \ \ \ \ \ \ \ \ \ \ \ \ \ - (x+1/\phi^2)^2 / (2\phi)+ 1/(2\phi^4),\\
\begin{array}{clll}
&\Longleftrightarrow&- x^2 /( 2\phi)-x/\phi-1/(2\phi^3)  -x/\phi+1/\phi^2-1/(2\phi^3)=-x \\
&& \ \ \ \ \ \ \ \ \ \ \ \ \ \ \ \ \ \ \ \ \ \ \ \ \ \ \ \ \ \ \ \ \ \ \ \ \ \ \ \ \ \ \ \ \ \ \ \ \ \ \ \ \ \ \ \ \ \ \ \ - (x+1/\phi^2)^2 / (2\phi)+ 1/(2\phi^2),\\
&\Longleftrightarrow& -x/\phi-1/(2\phi^3)  -x/\phi+1/\phi^2-1/(2\phi^3)=-x - x/ \phi^3-1/(2\phi^5)+ 1/(2\phi^2),\\
&\Longleftrightarrow&  \left\{\begin{array}{clclc}-2/\phi  &= &-1-1/\phi^3 ,\\-1/\phi^3 +1/\phi^2 & =& -1/(2\phi^5)+ 1/(2\phi^2).\end{array}\right.
\end{array}
\end{array}
\]
Puisque $\phi^3-2\phi^2+1=0$, nous trouvons le résultat attendu. Nous venons donc de montrer que :
\begin{equation} \label{eq:zonesind} 
T_\phi(\psi^{-1}(D_1)) \subset D_2 ', \  T_\phi(\psi^{-1}(D_2)) \subset D_1 ' \mbox{ et } T_\phi(D_2')-(1,0) \subset D_1 '. 
\end{equation}

\textbf{Les ensembles $D_1'$ et $D_2 '$ sont presque sûrement disjoints.} 
Soit $\boldsymbol{x} \in D_1' \cap D_2 '$. Il existe donc $(x,y)\in D$, et $(x_1,y_1)\in D_1$, tels que $\boldsymbol{x} =\psi^{-1}(x,y) = T_\phi (\psi^{-1}(x_1,y_1))$. Donc :
\[
\begin{pmatrix} -x/ \phi \\ -y- x^2 /( 2\phi)+x/(2\phi^2)  \end{pmatrix} = \begin{pmatrix}  -x_1/ \phi +1/\phi^2 \\ -y_1- x_1^2 /( 2\phi)+x_1/(2\phi^2)-x_1/\phi-1/(2\phi^3) \end{pmatrix}.
\]
Donc $x=x_1-1/\phi$, et 
\[
\begin{array}{l}
-y- x^2 /( 2\phi)+x/(2\phi^2)=-y_1- x_1^2 /( 2\phi)+x_1/(2\phi^2)-x_1/\phi-1/(2\phi^3),\\
\begin{array}{cllll}
&\Longleftrightarrow& y =y_1+ x_1^2 /( 2\phi)-x_1/(2\phi^2)+x_1/\phi+1/(2\phi^3) -x^2 /( 2\phi)+x/(2\phi^2) ,\\
&\Longleftrightarrow&  y =y_1+ x_1^2 /( 2\phi)-x_1/(2\phi^2)+x_1/\phi+1/(2\phi^3) \\
&& \ \ \ \  \ \ \ \  \ \ \ \  \ \ \ \  \ \ \ \  \ \ \ \  \ \ \ \  \ \ \ \  \ \ \ \  \ \ \ \  \ \ \ \ -(x_1-1/\phi)^2 /( 2\phi)+(x_1-1/\phi)/(2\phi^2),\\
&\Longleftrightarrow&  y = y_1+x_1/\phi+1/(2\phi^3) +x_1/\phi^2-1/(2\phi^3) -1/(2\phi^3) , \\
&\Longleftrightarrow&  y =y_1+x_1-1/(2\phi^3).
\end{array}
\end{array}
\]
Nous trouvons  $(x,y)=T_\phi(x_1,y_1)-(1,0)$. Mais puisque $(x_1,y_1)\in D_1$, le point $\boldsymbol{x}_0 = T_\phi(x_1,y_1) $ appartient à l'ensemble $D$. Il existe alors deux points  $(x,y)$ et $\boldsymbol{x}_0$ de $D$ tels que la différence vaut $(-1,0)\in \mathbb Z^2$. Donc, $\mu(D_1 ' \cap D_2')=0$, ce qui termine cette partie de la preuve.\\

\textbf{Il existe un réel $z'$ tel que $\mathfrak p(D_1')\subset ]z'-1,z'+1/\phi^2]$ et $\mathfrak p(D_2')\subset]z'-1/\phi^2,z'+1/ \phi^2[$.} 
C'est immédiat en remarquant qu'en notant $h$ l'homothétie réelle de rapport $-1/\phi$, alors $\mathfrak p\circ \psi^{-1}=h\circ\mathfrak  p$.\\

\textbf{Donnons-nous deux points $\boldsymbol{x}_1$ et $\boldsymbol{x}_2$ de $D'$ et vérifions que presque sûrement : $\boldsymbol{x}_1-\boldsymbol{x}_2\notin \mathbb Z^2$.} 
Commençons par supposer qu'il existe  deux points $(x,y)\in D$ et $(x_1,y_1)$ de $D_1$, et un vecteur entier $\boldsymbol{n}=(n,m)\in \mathbb Z^2$, tels que $\psi^{-1}(x,y) = T_\phi \psi^{-1}(x_2,y_2)+(n,m)$. Alors :
\[
\begin{pmatrix} -x/ \phi \\-y- x^2 /( 2\phi)+x/(2\phi^2) \end{pmatrix} =\begin{pmatrix} -x_1/ \phi +1/\phi^2+n \\-y_1- x_1^2 /( 2\phi)+x_1/(2\phi^2)-x_1/\phi-1/(2\phi^3)+m \end{pmatrix}.
\]
Donc $x=x_1-1/\phi-n\phi$. Puisque $\mathfrak p(D_1)\subset]z-1,z+1/\phi^2]$, et $\mathfrak p(D)\subset]z-1,z+1+1/\phi^2]$, alors nécessairement $n\in\{0,-1\}$. \\

Si $n=0$, nous pouvons directement  utiliser le fait que les ensembles $D_1'$ et $D_2 '$ sont presque sûrement disjoints. Sinon, $n=-1$, et donc $x=x_1+1$. Alors :
\[
\begin{array}{ll}
-y- x^2 /( 2\phi)+x/(2\phi^2) = -y_1- x_1^2 /( 2\phi)+x_1/(2\phi^2)-x_1/\phi-1/(2\phi^3)+m,\\
\begin{array}{cllllcccc}
&\Longleftrightarrow& y = y_1+ x_1^2 /( 2\phi)-x_1/(2\phi^2)+x_1/\phi+1/(2\phi^3) -x^2 /( 2\phi)+x/(2\phi^2)+m,\\
&\Longleftrightarrow& y = y_1+ x_1^2 /( 2\phi)-x_1/(2\phi^2)+x_1/\phi+1/(2\phi^3) -(x_1+1)^2 /( 2\phi)\\
&          & \ \ \ \  \ \ \  \ \ \ \  \ \ \ \  \ \ \ \ \  \ \ \ \  \ \ \ \   \ \ \ \  \ \ \ \  \ \ \ \  \ \ \ \  \ \ \ \  \ \  \ \ \ \  \ \ \ \  \ \ \ \  \ \ \ \   \ \ +(x_1+1)/(2\phi^2)+m,\\
&\Longleftrightarrow& y = y_1+x_1/\phi+1/(2\phi^3) -x_1/\phi-1/(2\phi) +1/(2\phi^2)+m,\\
&\Longleftrightarrow& y =y_1+m.
\end{array}
\end{array}
\]

La mesure de ces points est nulle. Nous venons de traiter le cas le plus difficile. En effet, avec les mêmes notations, supposons que  $(\boldsymbol{x}^1,\boldsymbol{x}^2)$ est un couple de point de $D_1'$ ou $D_2'$, tel que $\boldsymbol{x}^1-\boldsymbol{x}^2=(n,m)\in \mathbb Z^2$. Alors, avec les contraintes imposées sur  les longueurs à $\mathfrak p(D)$ et $\mathfrak p(D_1)$, il est relativement immédiat que $n=0$. Nous en déduisons alors aisément qu'il existe deux points de $D$ dont la différence est un vecteur de la forme $(0,m)$, où $m$ est un entier. Ces points forment encore un ensemble de mesure nulle.\\

\textbf{Vérifions que $\mathcal L_{\mathcal{R}'} = \mathcal L(\sigma(\mathcal L_{\mathcal{R}}))$.} 
L'application de premier retour de $R'$ dans $D_1'$ est conjuguée via $\psi$ à $R$, par l'équation \ref{eq:zonesind}. Fixons alors $\boldsymbol{x}' \in D_1'$ et notons $\boldsymbol{x}$ le point de $D$, pré-image de $\boldsymbol{x}'$ par $\psi$. Nous trouvons ainsi : $C(\boldsymbol{x}')=\sigma(C(\boldsymbol{x}))$. Supposons maintenant que $\boldsymbol{x}' \in D_2'$ et notons $\boldsymbol{z}'=T_\phi^{-1} \boldsymbol{x}' \in D_1'$. Par le raisonnement précédent, nous savons qu'il existe $\boldsymbol{x}\in D$, tel que $C(\boldsymbol{z}')=\sigma(C(\boldsymbol{x}))$. Nous concluons alors en remarquant que $\mathcal S(C(\boldsymbol{z}'))=C(\boldsymbol{x}')$  et donc que $C(\boldsymbol{x}')=\mathcal S(\sigma(C(\boldsymbol{x})))$. \\

\textbf{Le système symbolique $\mathcal{S}_{\mathcal{R}'}$ est conjugué à $\overline{\mathcal{T}}_\phi$.} 
Donnons-nous deux points $\boldsymbol{x}^1$ et $\boldsymbol{x}^2$ de $D'$ et supposons qu'ils aient le même codage. Nous reprenons les arguments du point précédent. 

Si ces deux points appartiennent à l'ensemble $D_1'$, alors il existe deux points  $\boldsymbol{y}^1$ et $\boldsymbol{y}^2$ de $D$ tels que $C(\boldsymbol{x}^1)=\sigma(C(\boldsymbol{y}^1))$ et  $C(\boldsymbol{x}^2)=\sigma(C(\boldsymbol{y}^2))$. Puisque la substitution de Fibonacci est un automorphisme du groupe libre : $C(\boldsymbol{y}^1)=C(\boldsymbol{y}^2)$. Nous pouvons conclure car la mesure de ces points est nulle. Nous traitons le cas où les points $\boldsymbol{x}^1$ et $\boldsymbol{x}^2$ appartiennent à  $D'_2$ de manière analogue.
\end{proof}

\begin{proof}[Preuve de la proposition \ref{prop:phi1}]
Notons  $z=1/(2\phi)+1/(2\phi^3)$ et considérons les fonctions quadratiques $p$, $q$ et $r$, définies pour tout réel $x$, par :
\[
p (x)=\phi^2\frac{x^2}{2}-\phi \frac{x}{2}-\frac{1}{\phi}, \ q (x)=p(x)+\phi^2x +\frac{3}{2} \mbox{ et }r (x)=p(x)-\phi^2x +1+\frac{1}{2\phi^3}.
\]
Nous définissons les zones $D_1^{(1)}$ et $D_2^{(1)}$ par :
\[
\begin{array}{clclcl} & &D_1^{(1)} & = & \Big{\{} \ (x,y)\ ;\ p (x)<y\leq p (x)+1 \mbox{ et } y \leq \min(q (x),r (x)-1) \  \Big{\}} \\
& \mbox{ et }& D_2^{(1)} & = & \Big{\{} \ (x,y)\ ;\ p (x)<y\leq p(x)+1 \mbox{ et } r (x)-1<y \leq r (x) \  \Big{\}}  . \end{array}
\]

Nous notons $R^{(1)}$ l'application de l'ensemble $D$ dans lui-même définie par 
\[
R^{(1)}(\boldsymbol{x})=T_\phi(\boldsymbol{x}) \mbox{, si } \boldsymbol{x}\in D_1 \mbox{, et  } R^{(1)}(\boldsymbol{x})=T_\phi(\boldsymbol{x})-(1,0) \mbox{ si }\boldsymbol{x}\in D_2.
\]

\begin{figure}[H] \centering
 \includegraphics[width=10cm]{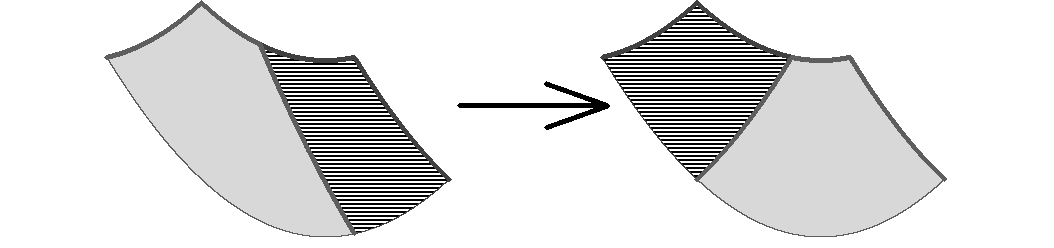} \caption{Représentation de l'application $R^{(1)}$ sur $D^{(1)}=D_1^{(1)} \cup D_2^{(1)}$.} \label{fig:e1}
\end{figure}

Nous montrons alors que le système symbolique $\mathcal{S}_{\mathcal{R}}$ est conjugué à $\overline{\mathcal{T}}_\phi$ en utilisant la proposition \ref{prop:raff}. Soit $\boldsymbol{x}^1=(x_1,y_1)$ et $\boldsymbol{x}^2=(x_2,y_2)$ deux points de $D^{(1)}$ tels que $C(\boldsymbol{x}^1)=C(\boldsymbol{x}^2)$. Commençons par supposer que $x_1=x_2-\delta$ avec $\delta>0$. Notons $\mathfrak p$ la projection $(x,y)\mapsto x$ et remarquons que pour tout couple $(\boldsymbol{y},\boldsymbol{y}')$ de $D^{(1)}_1$ (respectivement $D^{(1)}_2$), $\mathfrak p(\boldsymbol{y})-\mathfrak p(\boldsymbol{y}')= \mathfrak p\left(R^{(1)} \boldsymbol{y}\right)-\mathfrak p\left(R^{(1)} \boldsymbol{y}'\right)$. Donc si $C(\boldsymbol{x}^1)=C(\boldsymbol{x}^2)$, 
\[
\mbox{pour tout entier $n$,  } \mathfrak p \left( {R^{(1)}}^{n } (\boldsymbol{x}^2) \right) - \mathfrak p \left( {R^{(1)}}^{n } (\boldsymbol{x}^1)\right)  =\delta.
\]

Par minimalité du système $\overline{\mathcal{T}}_\phi$, il existe un entier $n_0$ tel que :
\[
{R^{(1)}}^{n_0}(\boldsymbol{x}^2)\in ]z-1/(\phi^2),z-1/(\phi^2)+\delta/2[\times \mathbb R \cap D^{(1)}_2.
\]
Nous arrivons à une contradiction puisque $\{z-1/(\phi^2)-\delta/2 \}\times \mathbb R \cap D^{(1)}_2 = \emptyset$, donc nécessairement $x_1=x_2$. Nous allons appliquer le même raisonnement pour montrer que $y_1=y_2$. Notons $x=x_1=x_2$, et supposons que $y_1=y_2+\delta$, où $\delta >0$. Notons $\mathfrak q$ la projection $(x,y)\mapsto y$. Comme précédemment, si $C(\boldsymbol{x}^1)=C(\boldsymbol{x}^2)$,  
\[
\mbox{ pour tout entier $n$,  } \mathfrak q  \left( {R^{(1)}}^{n_0 } (\boldsymbol{x}^1) \right) - \mathfrak q \left( {R^{(1)}}^{n } (\boldsymbol{x}^2)\right)  =\delta.
\]

Il existe $\epsilon>0$, et un entier $n_0$ tels que
\[
{R^{(1)}}^{n_0}(\boldsymbol{x}^1)\in ]z,z-\epsilon [\times \mathbb R \cap D^{(1)}_1,
\]
et donc l'ensemble $]z ,z-\epsilon +\delta [\times \mathbb R \cap D^{(1)} + (0,\delta)$ est disjoint de $D^{(1)}_1$.
\end{proof}

\begin{proof}[Preuve du théorème \ref{th:2}] Nous avons déjà vu dans la proposition \ref{prop:n+1}, que pour tout entier $M$, $p^{T_\phi} (M) \geq M+1$. Nous allons montrer que pour tout entier $M$ fixé, nous pouvons construire un échange de morceaux $\mathcal R^{(N)}$ conjugué à $T_\phi$, tel que pour tout entier $1\leq k \leq M$, $p_{\mathcal R^{(N)}} (k)=k+1$. \\

Nous définissons pour tout entier $N$ par récurrence un échange de morceaux  $\mathcal{R}^{(N)}$ conjugué à $T_\phi$ et satisfaisant les hypothèses de la proposition \ref{prop:phi2}. Nous initialisons la récurrence en considérant l'échange de morceaux $\mathcal{R}^{(1)}$ construit dans la proposition \ref{prop:phi1}. Le $(N+1)$-ième échange de morceaux est celui obtenu par la proposition \ref{prop:phi2} appliquée à l'échange de morceaux  $\mathcal{R}^{(N)}$.

\begin{figure}[H]  \centering  \includegraphics[width=14cm]{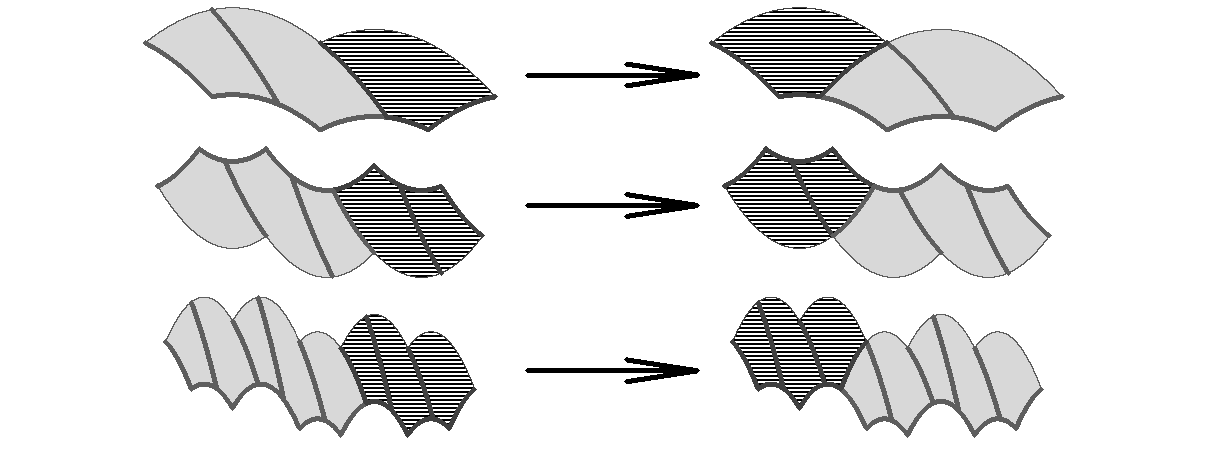}
\caption{Représentation des échanges de morceaux $\mathcal{R}^{(2)}$, $\mathcal{R}^{(3)}$ et $\mathcal{R}^{(4)}$.}
\end{figure}

Commençons donc par fixer un entier $M$. Nous notons $\mathfrak L^\phi$ le langage factoriel généré par le mot de Fibonacci, point fixe de la substitution $\sigma$. 
Par l'équation (\ref{eq:langa}) et le lemme suivant, la suite de langages $(\mathfrak L_{\mathcal R^{(N)}})_N$ associés aux échanges de morceaux $(\mathcal R^{(N)})_N$ converge au sens de l'équation (\ref{eq:con}) vers le langage $\mathfrak L^\phi$, qui est de complexité $M \mapsto M+1$.
\end{proof}

Nous énonçons maintenant un lemme combinatoire dont la preuve nous a été communiquée par J. Cassaigne.

\begin{lemme} \label{le:julien}
Soit $\sigma$ la substitution de Fibonacci définie en (\ref{eq:fibo}). Nous notons $\mathfrak L^\phi$ le langage factoriel généré par le mot de Fibonacci, point fixe de la substitution $\sigma$. 
Nous définissons pour tout langage factoriel $\mathfrak L$ de $\{1,2\}^*$ une suite de langages $(\mathfrak L_N)_N$ par récurrence de la manière suivante :
\begin{equation} \label{eq:lala}
\mathfrak L_0= \mathfrak L \mbox{ et pour tout entier $n \geq0$, } \mathfrak L_{N+1} = \mathcal L\big{(} \sigma(\mathfrak L_N) \big{)}. 
\end{equation}
Alors la suite de langages $(\mathfrak L_N)_N$ converge vers $\mathfrak L^\phi$, dans le sens où pour tout entier $M$, il existe un rang $N_0$ à partir duquel pour tout $N\geq N_0$,
\begin{equation} \label{eq:con}
\{ u \in  \mathfrak L_N  \mbox{ de longueur inférieure à $M$ } \} =\{ u \in \mathfrak L^\phi \mbox{ de longueur inférieure à $M$ } \}.
\end{equation} 
 \end{lemme}

\begin{proof}[Preuve du lemme \ref{le:julien}]
Donnons-nous deux langages  $\mathfrak L$ et $\mathfrak L'$, et associons-leur les suites de langages définis en (\ref{eq:lala}), respectivement notées $(\mathfrak L_N)_N$ et $(\mathfrak L_N')_N$. Nous commençons par remarquer que si $\mathfrak L$ est  inclus dans $\mathfrak L'$, alors pour tout entier $N$ le langage $\mathfrak L_N$   est inclus dans  $\mathfrak L_N'$. Il suffit donc de traiter les deux cas extrêmes : $\mathfrak L=\{1,2,\epsilon\}$ et $\mathfrak L=\{1,2\}^*$. 

Dans le premier cas, la suite $(\mathfrak L_N)_N$ est croissante (pour l'inclusion) ; la limite des langages $(\mathfrak L_N)_N$ est donc leur réunion qui est clairement le langage $\mathfrak L^\phi$ du mot de Fibonacci. 

Dans le second cas, la suite de langages $(\mathfrak L_N)_N$ est décroissante et sa limite est son intersection qui contient clairement le langage $\mathfrak L^\phi$. L'inclusion réciproque est plus délicate : c'est le seul endroit où nous utilisons des propriétés spécifiques de la substitution de Fibonacci. Si $w$ est dans l'intersection des langages $(\mathfrak L_N)_N$, alors en particulier $w$ est facteur de mots de la forme $\sigma^m(v_m)$, où $m$ est un entier et $v_m$ un mot quelconque. Nous pouvons alors supposer que $m$ est un entier non nul,  tel que $|s^m(1)| \geq |s^m(2)| \geq |w|$ et  que $v_m$ un mot  de longueur au plus $2$. Si $v_m$ n'est pas le mot $22$, $v_m$ est dans le langage $\mathfrak L^\phi$ (qui contient clairement les mots $11$, $12$ et $22$) donc $w$ aussi. Si $v_m=22$, alors $w$ est facteur de $\sigma^m(22) = \sigma^{m-1}(11)$ : il est donc  également un mot du langage $\mathfrak L^\phi$.
\end{proof}

\begin{figure}[H] \centering \includegraphics[width=12cm]{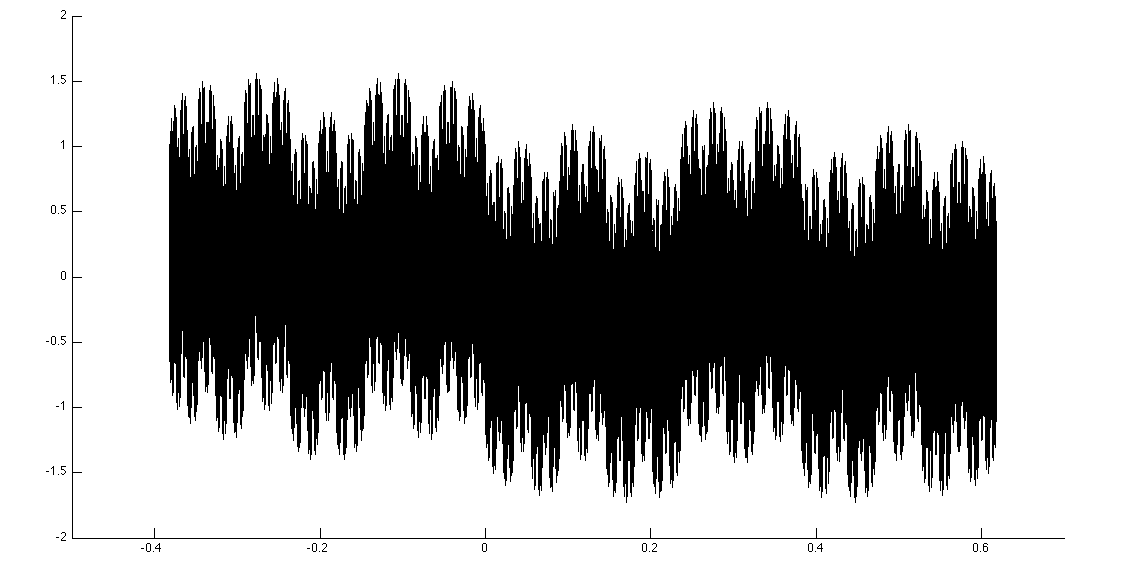} \caption{Représentation du domaine $D^{(20)}$.} \end{figure}

\begin{remarque} \label{rema}
Expliquons pourquoi la suite d'échanges de morceaux $(\mathcal R^{(N)})_N$ ne converge pas vers un «échange de morceaux bornés limite». \\
\indent
Si cet échange de morceaux  limite $\mathcal R^{(\infty)}$ existait, il serait un échange affine de deux morceaux $D^{(\infty)}_1$ et $D^{(\infty)}_2$  
dont l'image par la projection $\mathfrak p(x,y)=x$ serait respectivement $[-1/\phi,0]$ et $]0,1/\phi^2[$.\\
\indent
Pour tout point $\boldsymbol{x}_0=(x_0,y_0)$ de $D^{(\infty)}=D^{(\infty)}_1\cup D^{(\infty)}_2$,  nous nous intéressons à la suite de points $(\boldsymbol{x}_n)_n$ définie par recurrence par $\boldsymbol{x}_{n+1} = \mathcal R^{(\infty)} (\boldsymbol{x}_n)$.\\
\indent
Si nous notons $\{z\}$ la partie fractionnaire d'un nombre réel $z$,  alors par construction de la suite d'échanges de morceaux $(\mathcal R^{(N)})_N$ construits dans la preuve du théorème \ref{th:2}, nous trouvons :
\[
\begin{array}{cllll}
\boldsymbol{x}_{n} &=& \left( \left\{ x_0 + \frac{n}{\phi^2}\right\} -\frac{1}{\phi},y_0 + \frac{n}{2\phi^3} + \sum \limits_{k=0}^{n} \big{(}  \left\{ x_0 + \frac{k}{\phi^2}\right\}  -\frac{1}{\phi} \big{)} \right),\\
&=& \left( \left\{ x_0 + \frac{n}{\phi^2}\right\} -\frac{1}{\phi},y_0 + \sum \limits_{k=0}^{n} \big{(} \left\{ x_0 + \frac{k}{\phi^2}\right\} -\frac{1}{2} \big{)} \right).
\end{array}
\]
Par le travail de C.G. Pinner dans \cite{MR1743497}, 
\[
\limsup \limits_N \left|   \frac{n}{2} - \sum \limits_{k=0}^{n}  \left\{ x_0 + \frac{k}{\phi^2}\right\}   \right|= + \infty.
\]
Donc le domaine $D^{(\infty)}$ ne serait pas borné.
\end{remarque}

\end{document}